\documentclass[fleqn]{article}

\usepackage{amsmath,amsthm,amssymb,latexsym}

\newtheorem{theorem}{Theorem}[section]

\newtheorem{lemma}[theorem]{Lemma}
\newtheorem{fact}[theorem]{Fact}
\newtheorem*{mthm}{Main Theorem}
\newtheorem{cor}[theorem]{Corollary}

\newcommand{\mrm}{\mathrm}
\newcommand{\mbb}{\mathbb}
\newcommand{\mcal}{\mathcal}
\newcommand{\msf}{\mathsf}
\newcommand{\lchon}{\textrm{``}}
\newcommand{\rchon}{\textrm{''}}
\newcommand{\rst}{\!\upharpoonright\!}

\newcommand{\dom}{\mathop{\mathrm{dom}} \nolimits}

\newenvironment{renumerate}%
{\begin{enumerate}}{\end{enumerate}}
{\begin{enumerate}}{\end{enumerate}}
\newenvironment{aenumerate}%
{\begin{enumerate}}{\end{enumerate}}

\title{Martin's Maximum and the Diagonal Reflection Principle\footnote{
This research is supported by Simons Foundation Grant 318467 and JSPS Kakenhi Grant Number 18K03397.
}}
\author{Sean D.~Cox (Virginia Commonwealth University) \\
Hiroshi Sakai (Kobe University)}
\date{}

\begin{document}

\maketitle

\begin{abstract}
We prove that Martin's Maximum does not imply
the Diagonal Reflection Principle for stationary subsets of $[ \omega_2 ]^\omega$.
\end{abstract}

%%%%%%%%%%%%%%%%%%%%%%%%%%%%%%%%%%%%%%%%%%%%%%%%%%%%%%%%%%%%
%%%%%%%%%%%%%%%%%%%%%%%%%%%%%%%%%%%%%%%%%%%%%%%%%%%%%%%%%%%%

\section{Introduction} \label{sec:intro}

In Foreman-Magidor-Shelah \cite{FMS}, it was shown that Martin's Maximum $\msf{MM}$
implies the following stationary reflection principle,
which is called the Weak Reflection Principle:
\begin{list}{}{\setlength{\labelwidth}{50pt}\setlength{\leftmargin}{55pt}}
\item[$\msf{WRP} \equiv$]
For any cardinal $\lambda \geq \omega_2$ and any stationary $X \subseteq [ \lambda ]^\omega$,
there is $R \in [ \lambda ]^{\omega_1}$ with $R \supseteq \omega_1$ such that
$X \cap [R]^\omega$ is stationary in $[R]^\omega$.
\end{list}
$\msf{WRP}$ is known to have many interesting cosequences such as
Chang's Conjecture (Foreman-Magidor-Shelah \cite{FMS}),
the presaturation of the non-stationary ideal over $\omega_1$ (Feng-Magidor \cite{FM}),
$2^\omega \leq \omega_2$ (folklore) and
the Singular Cardinal Hypothesis (Shelah \cite{Sh:RP_SCH}).

As for stationary reflection principles, simultaneous reflection is often discussed.
Larson \cite{Larson} proved that $\msf{MM}$ also implies the following simultaneous reflection principle
of $\omega_1$-many stationary sets:
\begin{list}{}{\setlength{\labelwidth}{55pt}\setlength{\leftmargin}{60pt}}
\item[$\msf{WRP}_{\omega_1} \equiv$]
For any cardinal $\lambda \geq \omega_2$ and any sequence $\langle X_\xi \mid \xi < \omega_1 \rangle$
of stationary subsets of $[ \lambda ]^\omega$,
there is $R \in [ \lambda ]^{\omega_1}$ with $R \supseteq \omega_1$ such that
$X_\xi \cap [R]^\omega$ is stationary in $[R]^\omega$ for all $\xi < \omega_1$.
\end{list}

Cox \cite{Cox:DRP} formulated the following strengthening of $\msf{WRP}_{\omega_1}$,
which is called the Diagonal Reflection Principle:
\begin{list}{}{\setlength{\labelwidth}{50pt}\setlength{\leftmargin}{55pt}}
\item[$\msf{DRP} \equiv$]
For any cardinal $\lambda \geq \omega_2$ and any sequence $\langle X_\alpha \mid \alpha < \lambda \rangle$
of stationary subsets of $[ \lambda ]^\omega$,
there is $R \in [ \lambda ]^{\omega_1}$ with $R \supseteq \omega_1$ such that
$X_\alpha \cap [R]^\omega$ is stationary in $[R]^\omega$ for all $\alpha \in R$.
\end{list}
Recently, Fuchino-Ottenbreit-Sakai \cite{FOS} proved that
a variation of $\msf{DRP}$ is equivalent to some variation of the downward L\"{o}wenheim-Skolem theorem
of the stationary logic.
Cox \cite{Cox:DRP} also introduced the following weakning of $\msf{DRP}$,
where $X \subseteq [ \lambda ]^\omega$ is said to be \emph{projectively stationary} if
the set $\{ x \in X \mid x \cap \omega_1 \in S \}$ is stationary in $[ \lambda ]^\omega$
for any stationary $S \subseteq \omega_1$:
\begin{list}{}{\setlength{\labelwidth}{50pt}\setlength{\leftmargin}{55pt}}
\item[$\msf{wDRP} \equiv$]
For any cardinal $\lambda \geq \omega_2$ and any sequence $\langle X_\alpha \mid \alpha < \lambda \rangle$
of projectively stationary subsets of $[ \lambda ]^\omega$,
there is $R \subseteq [ \lambda ]^{\omega_1}$ with $R \supseteq \omega_1$ such that
$X_\alpha \cap [R]^\omega$ is stationary in $[R]^\omega$ for all $\alpha \in R$.
\end{list}

Cox \cite{Cox:DRP} proved that $\msf{MM}$ implies $\msf{wDRP}$,
but it remained open whether $\msf{MM}$ implies $\msf{DRP}$.
In this paper, we prove that $\msf{MM}$ does not imply $\msf{DRP}$.
In fact, we prove slightly more.

To state our result, we recall $+$-versions of the forcing axiom.
For a class $\Gamma$ of forcing notions and a cardinal $\mu \leq \omega_1$,
$\msf{MA}^{+ \mu} ( \Gamma )$ is the following statement:
\begin{list}{}{\setlength{\labelwidth}{65pt}\setlength{\leftmargin}{65pt}}
\item[$\msf{MA}^{+ \mu} ( \Gamma ) \equiv$]
For any $\mbb{P} \in \Gamma$, any sequence $\langle D_\xi \mid \xi < \omega_1 \rangle$
of dense subsets of $\mbb{P}$ and any sequence $\langle \dot{S}_\eta \mid \eta < \mu \rangle$ of
$\mbb{P}$-names of stationary subsets of $\omega_1$, there is a filter $g \subseteq \mbb{P}$
such that
\begin{renumerate}
\item $g \cap D_\xi \neq \emptyset$ for any $\xi < \omega_1$,
\item $\dot{S}_\eta^g =
\{ \alpha < \omega_1 \mid \exists p \in g , \ p \Vdash_\mbb{P} \lchon \alpha \in \dot{S}_\eta \rchon \}$
is stationary in $\omega_1$ for all $\eta < \mu$.
\end{renumerate}
\end{list}
Let $\msf{MA}^{+ \mu} ( \sigma\mbox{-closed} )$ denote $\msf{MA}^{+ \mu} ( \Gamma )$ for
the class $\Gamma$ of all $\sigma$-closed forcing notions.
Also, let $\msf{MM}^{+ \mu}$ denote $\msf{MA}^{+ \mu} ( \Gamma )$ for the class $\Gamma$
of all $\omega_1$-stationary preserving forcing notions.
It is well-known that $\msf{MA}^{+ \omega_1} ( \sigma\mbox{-closed} )$ holds if a supercompact
cardinal is L\'{e}vy collapsed to $\omega_2$ and that $\msf{MM}^{+ \omega_1}$ holds
in the standard model of $\msf{MM}$ constructed in Foreman-Magidor-Shelah \cite{FMS}.

Cox \cite{Cox:DRP} proved that $\msf{MA}^{+ \omega_1} ( \sigma\mbox{-closed} )$ implies
$\msf{DRP}$. So $\msf{MM}^{+ \omega_1}$ implies $\msf{DRP}$.
In this paper, we prove that $\msf{MM}^{+ \omega}$ does not imply $\msf{DRP}$:

\begin{mthm}
Assume $\msf{MM}^{+ \omega}$ holds.
Then there is a forcing extension in which $\msf{MM}^{+ \omega}$ remains to hold,
but $\msf{DRP}$ fails at $[ \omega_2 ]^\omega$.
\end{mthm}

Our proof of the Main Theorem is based on the proof of the classical result, due to Beaudoin \cite{Beaudoin} and
Magidor, that the Proper Forcing Axiom does not imply the reflection of
stationary subsets of the set $\{ \alpha \in \omega_2 \mid \mrm{cof} ( \alpha ) = \omega \}$.
Similar arguments are used in K\"{o}nig-Yoshinobu \cite{KY}, Yoshinobu \cite{Y1}, \cite{Y2}
and Cox \cite{Cox:sep}, to separate reflection principles from strong forcing axioms.

We will prove the Main Theorem in Section \ref{sec:proof}.
In Section \ref{sec:preliminaries}, we will present our notation and basic facts
used in this paper.

%%%%%%%%%%%%%%%%%%%%%%%%%%%%%%%%%%%%%%%%%%%%%%%%%%%%%%%%%%%%
%%%%%%%%%%%%%%%%%%%%%%%%%%%%%%%%%%%%%%%%%%%%%%%%%%%%%%%%%%%%

\section{Preliminaries} \label{sec:preliminaries}

Here we present our notation and basic facts.
See Jech \cite{Jech} for those which are not mentioned here.

First, we recall the notion of stationary sets in $[W]^\omega$.
Let $W$ be a set with $\omega_1 \subseteq W$.
$Z \subseteq [W]^\omega$ is said to be \emph{club} in $[W]^\omega$
if $Z$ is $\subseteq$-cofinal in $[W]^\omega$,
and $\bigcup_{n \in \omega} x_n \in Z$ for any $\subseteq$-increasing sequence
$\langle x_n \mid n < \omega \rangle$ of elements of $Z$.
$X \subseteq [W]^\omega$ is said to be \emph{stationary} in $[W]^\omega$
if $X \cap Z \neq \emptyset$ for any club $Z \subseteq [W]^\omega$.
For $S \subseteq \omega_1$,
$S$ is stationary in $\omega_1$ in the usual sense
if and only if $S$ is stationary in $[ \omega_1 ]^\omega$ in the above sense.

We will use the following standard facts without any reference.
Proofs can be found also in Jech \cite{Jech}.

\begin{fact}[(1) Kueker \cite{Kueker}, (2) Menas \cite{Menas}] \label{fact:stat_basic}
Suppose $W$ is a set $\supseteq \omega_1$ and $X$ is a subset of $[W]^\omega$.
\begin{aenumerate}
\item $X$ is stationary if and only if for any function $F : [W]^{< \omega} \to W$
there is a non-empty $x \in X$ which is closed under $F$, i.e.~$F(a) \in x$ for all $a \in [x]^{< \omega}$.
\item Suppose $W' \supseteq W$. Then $X$ is stationary in $[W]^\omega$ if and only if
the set $\{ x' \in [W']^\omega \mid x' \cap W \in X \}$ is stationary in $[W']^\omega$.
\end{aenumerate}
\end{fact}

Here we slightly simplify $\msf{DRP}$ at $[ \omega_2 ]^\omega$.

\begin{lemma} \label{lem:DRP_easy}
Assume $\msf{DRP}$ at $[ \omega_2 ]^\omega$.
Then, for any sequence $\langle X_\alpha \mid \alpha < \omega_2 \rangle$ of
stationary subsets of $[ \omega_2 ]^\omega$, there is $\delta \in \omega_2 \setminus \omega_1$
such that $X_\alpha \cap [ \delta ]^\omega$ is stationary in $[ \delta ]^\omega$
for all $\alpha < \delta$.
\end{lemma}

\begin{proof}
Suppose $\langle X_\alpha \mid \alpha < \omega_2 \rangle$ is a sequence of stationary
subsets of $[ \omega_2 ]^\omega$. We find $\delta$ as in the lemma.

For each $\beta < \omega_2$, take a surjection $\pi_\beta : \omega_1 \to \beta$.
Let $Z$ be the set of all $x \in [ \omega_2 ]^\omega$ such that
$x \cap \omega_1 \in \omega_1$ and $x$ is closed under $\pi_\beta$ for all $\beta \in x$.
Then, $Z$ is club in $[ \omega_2 ]^\omega$.
Moreover, it is easy to see that if $\omega_1 \subseteq R \in [ \omega_2 ]^{\omega_1}$,
and $Z \cap [R]^\omega$ is $\subseteq$-cofinal in $[R]^\omega$,
then $R \in \omega_2 \setminus \omega_1$.

By shrinking $X_0$ if necessary, we may assume that $X_0 \subseteq Z$.
By $\msf{DRP}$ at $[ \omega_2 ]^\omega$,
take $R \in [ \omega_2 ]^{\omega_1}$ including $\omega_1$
such that $X_\alpha \cap [R]^\omega$ is stationary for all $\alpha \in R$.
Then, $R \in \omega_2 \setminus \omega_1$ since $Z \cap [R]^\omega$ is $\subseteq$-cofinal
in $[R]^\omega$. So, $\delta := R$ is as desired.
\end{proof}

Next, we present our notation and basic facts about forcing.
Suppose $\mbb{P}$ is a forcing notion and $M$ is a set.
We say that $g \subseteq \mbb{P} \cap M$ is $M$-\emph{generic} if
$g \cap D \neq \emptyset$ for any dense $D \subseteq \mbb{P}$ with $D \in M$.

We will use the following well-known fact about forcing axioms:

\begin{fact}[Woodin \cite{Woodin}] \label{fact:forcing_axiom}
Let $\Gamma$ be a class of forcing notions and $\mu$ be a cardinal $\leq \omega_1$,
and assume $\msf{MA}^{+ \mu} ( \Gamma )$ holds.
Suppose $\mbb{P} \in \Gamma$ and $\langle \dot{T}_\xi \mid \xi < \mu \rangle$
is a sequence of $\mbb{P}$-names for stationary subsets of $\omega_1$.
Then, for any regular cardinal $\theta$
with $\mbb{P} \in \mcal{H}_\theta$ and any $A \in [ \mcal{H}_\theta ]^{\omega_1}$,
there are $M \in [ \mcal{H}_\theta ]^{\omega_1}$ and $g \subseteq \mbb{P} \cap M$
with the following properties.
\begin{renumerate}
\item $A \subseteq M \prec \langle \mcal{H}_\theta , \in \rangle$.
\item $g$ is an $M$-generic filter on $\mbb{P} \cap M$.
\item $\dot{T}_\xi^g$ is stationary in $\omega_1$ for any $\xi < \mu$.
\end{renumerate}
\end{fact}

We will also use forcing notions for shooting club sets.
For an ordinal $\lambda \geq \omega_1$ and a subset $X$ of $[ \lambda ]^\omega$,
let $\mbb{R} ( X )$ denote the poset of all $\subseteq$-increasing continuous
function from some countable successor ordinal to $X$,
which is ordered by reverse inclusions.
The following is standard:

\begin{lemma} \label{lem:club_shoot_basic}
Suppose $X$ is a stationary subset of $[ \lambda ]^\omega$ for some ordinal
$\lambda \geq \omega_1$.
\begin{aenumerate}
\item A forcing extension by $\mbb{R} ( X )$ adds no new countable sequences of ordinals.
So it preserves $\omega_1$.
\item In $V^{\mbb{R} (X)}$, $X$ contains a club subset of $[ \lambda ]^\omega$.
\item In $V$, suppose $Y \subseteq X$ and $Y$ is stationary in $[ \lambda ]^\omega$.
Then $Y$ remains stationary in $V^{\mbb{R} (X)}$.
\end{aenumerate}
\end{lemma}

\begin{proof}
Let $\mbb{R}$ denote $\mbb{R} (X)$.
Before starting, note that the set
$\{ r \in \mbb{R} \mid \exists \xi \in \dom (r) , \  r( \xi ) \supseteq x \}$
is dense in $\mbb{R}$ for any $x \in [ \lambda ]^\omega$,
since $X$ is $\subseteq$-cofinal in $[ \lambda ]^\omega$.

First, we prove (1) and (3). We work in $V$.
Suppose $r \in \mbb{R}$, $\mcal{D}$ is a countable family of dense open subsets of $\mbb{R}$
and $\dot{F}$ is an $\mbb{R}$-name for a function from $[ \lambda ]^{< \omega}$ to $\lambda$.
It suffices to find $r^* \leq r$ and $y \in Y$ such that
$r^* \in \bigcap \mcal{D}$ and $r^*$ forces $y$ to be closed under $\dot{F}$.

Take a sufficiently large regular cardinal $\theta$.
Since $Y$ is stationary, there is a countable $M \prec \langle \mcal{H}_\theta , \in \rangle$
such that $\{ \lambda , X , r , \dot{F} \} \cup \mcal{D} \subseteq M$ and
$y := M \cap \lambda \in Y$.
Then, we can construct a descending sequence
$\langle r_n \mid n < \omega \rangle$ in $\mbb{R} \cap M$ such that
$r_0 = r$ and $\{ r_n \mid n < \omega \}$ is $M$-generic.
Note that any lower bound of $\{ r_n \mid n < \omega \}$ forces $y$ to be closed
under $\dot{F}$ by the $M$-genericity of $\{ r_n \mid n < \omega \}$.

Let $r' := \bigcup_{n < \omega} r_n$ and $\zeta : = \dom ( r' )$.
Then, using the fact mentioned at the beginning,
it is easy to check that $\zeta$ is a limit ordinal and
$\bigcup_{\xi < \zeta} r' ( \xi ) = y$.
Let $r^*$ be an extension of $r'$ such that
$\dom ( r^* ) = \zeta + 1$ and $r^* ( \zeta ) = y$.
Then $r^* \in \mbb{R}$, and $r^*$ is a lower bound of $\{ r_n \mid n < \omega \}$.
So $r^*$ and $y$ are as desired.

Next, we check (2). By (1), the definition of $\mbb{R}$ and the fact mentioned at the beginning,
if $G$ is an $\mbb{R}$-generic filter over $V$, then $\mrm{range} ( \bigcup G )$
is a club subset of $[ \lambda ]^\omega$ consisting of elements of $X$.
So (2) holds.
\end{proof}

%%%%%%%%%%%%%%%%%%%%%%%%%%%%%%%%%%%%%%%%%%%%%%%%%%%%%%%%%%%%
%%%%%%%%%%%%%%%%%%%%%%%%%%%%%%%%%%%%%%%%%%%%%%%%%%%%%%%%%%%%

\section{Proof of Main Theorem} \label{sec:proof}

Here we prove the Main Theorem.
Throughout this section, assume that $\msf{MM}^{+ \omega}$ holds in the ground model $V$.

We construct a forcing notion which preserves $\msf{MM}^{+ \omega}$
and adds a counter-example $\langle X_\alpha \mid \alpha < \omega_2 \rangle$ of the consequence
of Lemma \ref{lem:DRP_easy}.
Here recall that $\msf{MM}$ implies $\msf{wDRP}$.
So we must arrange our forcing notion so that each $X_\alpha$ is not projectively stationary.
For some technical reason, we also make $\langle X_\alpha \mid \alpha < \omega_2 \rangle$
pairwise disjoint.

Recall the fact, due to Foreman-Magidor-Shelah \cite{FMS}, that $\msf{MM}$ implies
$2^{\omega_1} = \omega_2$. In $V$, fix an enumeration $\langle S_\alpha \mid \alpha < \omega_2 \rangle$
of all stationary subsets of $\omega_1$.
Let $\mbb{P}$ be the following forcing notion:
\begin{itemize}
\item $\mbb{P}$ consists of all functions $p$ such that
\begin{renumerate}
\item $p : \delta_p \times [ \delta_p ]^\omega \to 2$ for some $\delta_p < \omega_2$,
\item for any $\alpha < \delta_p$,
$X_{p,\alpha} := \{ x \in [ \delta_p ]^\omega \mid p( \alpha , x ) = 1 \}$
has size $\leq \omega_1$,
\item $x \cap \omega_1 \in S_\alpha$ for any $\alpha < \delta_p$ and any $x \in X_{p, \alpha}$,
\item $X_{p, \alpha} \cap X_{p , \beta} = \emptyset$ for any distinct $\alpha , \beta < \delta_p$,
\item for any $\delta \in \delta_p + 1 \setminus \omega_1$, there is $\alpha < \delta$ with
$X_{p , \alpha} \cap [ \delta ]^\omega$ non-stationary in $[ \delta ]^\omega$.
\end{renumerate}
\item $p \leq p'$ in $\mbb{P}$ if $p \supseteq p'$.
\end{itemize}

We observe basic properties of $\mbb{P}$.
Note that a forcing extension by $\mbb{P}$ preserves all cardinals by (1) and (3) of the
following lemma.

\begin{lemma} \label{lem:P_basic}
\begin{aenumerate}
\item $| \mbb{P} | = \omega_2$.
\item $\mbb{P}$ is $\sigma$-closed.
\item A forcing extension by $\mbb{P}$ adds no new sequences of ordinals of length $\omega_1$.
\item For any $p \in \mbb{P}$ and any $\delta < \omega_2$, there is $p' \leq p$ with
$\delta \leq \delta_{p'}$.
\end{aenumerate}
\end{lemma}

\begin{proof}
(1) This is clear from the definition of $\mbb{P}$, especially the property (ii) of its conditions,
and the fact that $2^{\omega_1} = \omega_2$ in $V$.

\medskip

\noindent
(4) Suppose $p \in \mbb{P}$ and $\delta < \omega_2$.
We may assume $\delta_p \leq \delta$.
Let $p' : \delta \times [ \delta ]^\omega \to 2$ be an extension of $p'$ such that
$p' ( \alpha , x ) = 0$ for all $\langle \alpha , x \rangle \notin \delta_p \times [ \delta_p ]^\omega$.
It suffices to prove that $p' \in \mbb{P}$. We only check that $p'$
satisfies the property (v) of conditions of $\mbb{P}$. The other properties are easily checked.

Take an arbitrary $\gamma \in \delta + 1 \setminus \omega_1$.
We find $\alpha < \delta$ with $X_{p' , \alpha} \cap [ \gamma ]^\omega$ is non-stationary.
If $\gamma \leq \delta_p$, then we can find such $\alpha$ since
$p \in \mbb{P}$ and $p \subseteq p'$.
Suppose $\gamma > \delta_p$. Then $Z := [ \gamma ]^\omega \setminus [ \delta_p ]^\omega$ is
club in $[ \gamma ]^\omega$, and $X_{p' , \alpha} \cap Z = \emptyset$ for any $\alpha < \gamma$.
So any $\alpha < \gamma$ is as desired in this case.

\medskip

\noindent
(2) Suppose $\langle p_n \mid n < \omega \rangle$ is a descending sequence in $\mbb{P}$.
We find a lower bound $p^*$ of $\{ p_n \mid n < \omega \}$ in $\mbb{P}$.
We may assume that $\langle p_n \mid n < \omega \rangle$ is not eventually constant.

Let $\delta_n := \delta_{p_n}$ for each $n < \omega$.
Let $\delta^* := \bigcup_{n < \omega} \delta_n$,
and let $p^* : \delta^* \times [ \delta^* ]^\omega \to 2$ be an extension of
$\bigcup_{n \in \omega} p_n$ such that
$p^* ( \alpha , x ) = 0$ for all $\alpha < \delta^*$ and
all $x \in [ \delta^* ]^\omega \setminus \bigcup_{n \in \omega} [ \delta_n ]^\omega$.
Note that $X_{p^* , \alpha}$ is non-stationary in $[ \delta^* ]^\omega$ for any $\alpha < \delta^*$
since $Z := [ \delta^* ]^\omega \setminus \bigcup_{n < \omega} [ \delta_n ]^\omega$
is club in $[ \delta^* ]^\omega$ and $X_{p^* , \alpha} \cap Z = \emptyset$.
Then it is easy to see that $p^*$ is as desired.

\medskip

\noindent
(3) Suppose $p \in \mbb{P}$ and $\langle D_\xi \mid \xi < \omega_1 \rangle$
is a sequence of dense open subsets of $\mbb{P}$.
It suffices to find $p^* \leq p$ with $p^* \in \bigcap_{\xi < \omega_1} D_\xi$.

We recursively construct a strictly descending sequence $\langle p_\xi \mid \xi < \omega_1 \rangle$
in $\mbb{P}$ as follows. For each $\xi < \omega_1$, we let $\delta_\xi$ denote
$\delta_{p_\xi}$. First, let $p_0 := p$.
If $p_\xi$ has been taken, then take $p_{\xi + 1} < p_\xi$ with $p_{\xi + 1} \in D_\xi$.
Suppose $\xi$ is a limit ordinal $< \omega_1$ and $\langle p_\eta \mid \eta < \xi \rangle$
has been constructed. Then define $p_\xi$ as in the proof of (2). That is,
let $\delta_\xi := \bigcup_{\eta < \xi} \delta_\eta$, and let
$p_\xi : \delta_\xi \times [ \delta_\xi ]^\omega \to 2$ be an extension of
$\bigcup_{\eta < \xi} p_\eta$ such that
$p_\xi ( \alpha , x ) = 0$ for all $\alpha < \delta_\xi$ and all
$x \in [ \delta_\xi ]^\omega \setminus \bigcup_{\eta < \xi} [ \delta_\eta ]^\omega$.
Then $p_\xi$ is a lower bound of $\{ p_\eta \mid \eta < \xi \}$ in $\mbb{P}$.

We have constructed $\langle p_\xi \mid \xi < \omega_1 \rangle$.
Let $\delta^* := \sup_{\xi < \omega_1} \delta_\xi$ and $p^* := \bigcup_{\xi < \omega_1} p_\xi$.
Here note that $[ \delta^* ]^\omega = \bigcup_{\xi < \omega_1} [ \delta_\xi ]^\omega$.
So $p^* : \delta^* \times [ \delta^* ]^\omega \to 2$.
Note also that $X_{p^* , 0}$ is non-stationary in $[ \delta^* ]^\omega$
since
\[
Z := \{ x \in [ \delta^* ]^\omega \mid
\mbox{$x \subseteq \delta_\xi = \sup (x)$ for some limit $\xi < \omega_1$} \}
\]
is club in $[ \delta^* ]^\omega$ and that $X_{p^* , 0} \cap Z = \emptyset$
by the construction of $p_\xi$ for a limit $\xi < \omega_1$.
Then, it is easy to check that $p^*$ is as desired.
\end{proof}

Let $\dot{G}$ be the canonical $\mbb{P}$-name for a $\mbb{P}$-generic filter.
For $\alpha < \omega_2$, let $\dot{X}_\alpha$ be the $\mbb{P}$-name
for the set
\[
\{ x \in [ \omega_2 ]^\omega \mid \exists p \in \dot{G} , \ p( \alpha , x ) = 1 \} \, .
\]

\begin{lemma} \label{lem:X_alpha_stationary}
For each $\alpha < \omega_2$,
$\dot{X}_\alpha$ is stationary in $[ \omega_2 ]^\omega$ in $V^\mbb{P}$.
\end{lemma}

\begin{proof}
We work in $V$.
Take an arbitrary $\alpha < \omega_2$.
Suppose $p \in \mbb{P}$ and $\dot{F}$ is a $\mbb{P}$-name for a function
from $[ \omega_2 ]^{< \omega}$ to $\omega_2$.
It suffices to find $p^* \leq p$ and $x \in [ \omega_2 ]^{< \omega}$ such that
$p^* \Vdash_\mbb{P} \lchon\, x \in \dot{X}_\alpha \wedge
\mbox{$x$ is closed under $\dot{F}$} \,\rchon$.

Take a sufficiently large regular cardinal $\theta$
and a countable $M \prec \langle \mcal{H}_\theta , \in \rangle$
such that $\alpha , \mbb{P} , p , \dot{F} \in M$ and $M \cap \omega_1 = \alpha$.
Let $x := M \cap \omega_2$.
We can take a descending sequence $\langle p_n \mid n < \omega \rangle$
in $\mbb{P} \cap M$ such that $p_0 = p$ and $\{ p_n \mid n < \omega \}$ is $M$-generic.
Note that any lower bound of $\{ p_n \mid n < \omega \}$ forces $x$ to be closed under $\dot{F}$
by the $M$-genericity.
For each $n < \omega$, let $\delta_n := \delta_{p_n}$.
Note that $\delta_n \in M \cap \omega_2$ for each $n < \omega$ and that
$\delta^* := \sup_{n < \omega} \delta_n = \sup ( M \cap \omega_2 )$ by Lemma \ref{lem:P_basic} (4).

Let $p^* : \delta^* \times [ \delta^* ]^\omega \to 2$
be an extension of $\bigcup_{n < \omega} p_n$ such that
$p^* ( \alpha , x ) = 1$ and $p^* ( \beta , y ) = 0$ for any $\beta < \delta^*$ and
any $y \in [ \delta^* ]^\omega \setminus \bigcup_{n < \omega} [ \delta_n ]^\omega$
with $\langle \beta , y \rangle \neq \langle \alpha , x \rangle$.
Then, it is easy to check that $p^*$ and $x$ are as desired.
\end{proof}

The following is immediate
from Lemma \ref{lem:DRP_easy}, \ref{lem:P_basic}, \ref{lem:X_alpha_stationary}
and the property (v) of conditions of $\mbb{P}$:

\begin{cor} \label{cor:DRP_fail}
$\msf{DRP}$ at $\omega_2$ fails in $V^\mbb{P}$.
\end{cor}

We must show that $\mbb{P}$ preserves $\msf{MM}^{+ \omega}$.
The following lemma is a key:

\begin{lemma} \label{lem:triple_stat_pres}
Let $\dot{\mbb{Q}}$ be a $\mbb{P}$-name for an $\omega_1$-stationary preserving
forcing notion and $\langle \dot{T}_n \mid n < \omega \rangle$ be a sequence
of $\mbb{P} * \dot{\mbb{Q}}$-names for stationary subsets of $\omega_1$.
Then there is a $\mbb{P} * \dot{\mbb{Q}}$-name $\dot{\gamma}$ of an ordinal $< \omega_2^V$
such that if we let
\[
\mbb{S} :=
\mbb{P} * \dot{\mbb{Q}} * \mbb{R} ( [ \omega_2^V ]^\omega \setminus \dot{X}_{\dot{\gamma}} )
\, ,
\]
then all elements of $\{ S_\alpha \mid \alpha < \omega_2^V \} \cup \{ \dot{T}_n \mid n < \omega \}$
remain stationary in $V^\mbb{S}$.
\end{lemma}

\begin{proof}
Let $\lambda := \omega_2^V$.
Suppose $G*H$ is a $\mbb{P} * \dot{\mbb{Q}}$-generic filter over $V$.
In $V[G*H]$, let $X_\alpha := \dot{X}_\alpha^G$ for $\alpha < \lambda$
and $T_n := \dot{T}_n^{G*H}$ for $n < \omega$.
Moreover, let $\mbb{R}_\alpha$ denote $\mbb{R} ( [ \lambda ]^\omega \setminus X_\alpha )$
for $\alpha < \lambda$.
In $V[G*H]$, we find $\gamma < \lambda$ such that
$\mbb{R}_\gamma$ forces
all elements of $\{ S_\alpha \mid \alpha < \lambda \} \cup \{ T_n \mid n < \omega \}$
stationary.
Here note that all $S_\alpha$ and $T_n$ are stationary in $V[G*H]$
by the fact that $\mbb{P} * \dot{\mbb{Q}}$ is $\omega_1$-stationary preserving
and the assumption on $\langle \dot{T}_n \mid n < \omega \rangle$.

We work in $V[ G*H ]$.
For $S \subseteq \omega_1$, let
$\bar{S} := \{ x \in [ \lambda ]^\omega \mid x \cap \omega_1 \in S \}$.
For $X,Y \subseteq [ \lambda ]^\omega$, we write $X \subseteq^* Y$ if
$X \setminus Y$ is non-stationary in $[ \lambda ]^\omega$.
By Lemma \ref{lem:club_shoot_basic}, for $S \subseteq \omega_1$ and $\alpha < \lambda$,
$\mbb{R}_\alpha$ does not force $S \subseteq \omega_1$ stationary if and only if
$\bar{S} \subseteq^* X_\alpha$.

Since $\langle X_\alpha \mid \alpha < \lambda \rangle$ is pairwise disjoint,
for each $n < \omega$ there is at most one $\alpha < \lambda$ with $\bar{T}_n \subseteq^* X_\alpha$.
Since $| \lambda | \geq \omega_1$, we can take $\beta < \lambda$ such that
$\bar{T}_n \not\subseteq^* X_\beta$ for any $n < \omega$.
Then $\mbb{R}_\beta$ forces $T_n$ stationary for all $n < \omega$.
Thus, if $\mbb{R}_\beta$ also forces $S_\alpha$ stationary for all $\alpha < \lambda$,
then $\gamma := \beta$ is as desired.

Assume there is $\alpha < \lambda$ such that $\mbb{R}_\beta$ does not force $S_\alpha$
stationary. By replacing $\alpha$ with $\alpha '$ such that $S_{\alpha '} \subseteq S_\alpha$
if necessary, we may assume that $\alpha \neq \beta$.
Here note that $X_\alpha \subseteq \bar{S}_\alpha$ by the property (iii) of conditions of $\mbb{P}$.
Then, $X_\alpha \subseteq \bar{S}_\alpha \subseteq^* X_\beta$
and $X_\alpha \cap X_\beta = \emptyset$. Hence $X_\alpha$ is non-stationary
in $[ \lambda ]^\omega$. Thus $\mbb{R}_\alpha$ is $\omega_1$-stationary preserving,
and so $\gamma := \alpha$ is as desired.
\end{proof}

Now, we can prove that $\mbb{P}$ preserves $\msf{MM}^{+ \omega}$
by a similar argument as Beaudoin \cite{Beaudoin}:

\begin{lemma} \label{lem:pres_MM+omega}
$\msf{MM}^{+ \omega}$ holds in $V^\mbb{P}$.
\end{lemma}

\begin{proof}
Let $\dot{\mbb{Q}}$ be a $\mbb{P}$-name for an $\omega_1$-stationary preserving
foricng notion.
For each $\xi < \omega_1$, let $\dot{D}_\xi$ be a $\mbb{P}$-name
for a dense subset of $\dot{\mbb{Q}}$, and for each $n < \omega$, let
$\ddot{T}_n$ be a $\mbb{P}$-name for a $\dot{\mbb{Q}}$-name for a stationary subset of $\omega_1$.
Take an arbitrary $p_0 \in \mbb{P}$.
It suffices to find $p^* \leq p_0$ in $\mbb{P}$ such that
if $G$ is a $\mbb{P}$-generic filter over $V$ with $p^* \in G$,
then in $V[ G ]$ there is a filter $h \subseteq \mbb{Q}$ with the following properties:
\begin{renumerate}
\item $h \cap D_\xi \neq \emptyset$ for any $\xi < \omega_1$.
\item $\dot{T}_n^h$ is stationary in $\omega_1$ for all $n < \omega$.
\end{renumerate}
Here $\mbb{Q}$, $D_\xi$ and $\dot{T}_n$ denote
$\dot{\mbb{Q}}^G$, $\dot{D}_\xi^G$ and $\ddot{T}_n^G$, respectively.

First, we find $p^*$ as above. We work in $V$.
We identify each $\ddot{T}_n$ with a $\mbb{P} * \dot{\mbb{Q}}$-name.
Let $\dot{\gamma}$ and $\mbb{S}$ be as in Lemma \ref{lem:triple_stat_pres}.
Note that $\mbb{S}$ is $\omega_1$-stationary preserving
and each $\ddot{T}_n$ is stationary in $\omega_1$ in $V^\mbb{S}$.
Let $\dot{\mbb{R}}$ be a $\mbb{P} * \dot{\mbb{Q}}$-name for
$\mbb{R} ( [ \omega_2^V ]^\omega \setminus \dot{X}_{\dot{\gamma}} )$.

Take a sufficiently large regular cardinal $\theta$.
By Fact \ref{fact:forcing_axiom}, there are $M \in [ \mcal{H}_\theta ]^{\omega_1}$
and $k \subseteq \mbb{S} \cap M$ such that
\begin{renumerate}
\addtocounter{enumi}{2}
\item $\omega_1 \cup \{ p_0 , \dot{\mbb{Q}} , \dot{\gamma} , \mbb{S} \} \cup
\{ \dot{D}_\xi \mid \xi < \omega_1 \} \cup \{ \ddot{T}_n \mid n < \omega \}
\subseteq M \prec \langle \mcal{H}_\theta , \in \rangle$,
\item $k$ is an $M$-generic filter on $\mbb{S} \cap M$
with $p_0 * 1_{\dot{\mbb{Q}}} * 1_{\dot{\mbb{R}}} \in k$,
\item $\ddot{T}_n^k$ is stationary in $\omega_1$ for any $\xi < \mu$.
\end{renumerate}
Let $\delta^* := M \cap \omega_2 \in \omega_2$, and let
\[
g := \{ p \in \mbb{P} \cap M \mid \exists \dot{q} \, \exists \dot{r} , \ p * \dot{q} * \dot{r} \in k \} \, .
\]
Then, $g$ is an $M$-generic filter on $\mbb{P} \cap M$.

Note that $\sup_{p \in g} \delta_p = \delta^*$
by Lemma \ref{lem:P_basic} (4) and the $M$-genericity of $g_0$.
Let $p^* : \delta^* \times [ \delta^* ]^\omega \to 2$ be an extension of $\bigcup g$
such that $p^* ( \alpha , x ) = 0$ for all $\langle \alpha , x \rangle \notin \dom ( \bigcup g )$.
We claim that $p^*$ is as desired.
For this, we use the transitive collapse of $M$.
First, we make some preliminaries on it.

Let $\pi : M \to M'$ be the transitive collapse of $M$, and let
$\mbb{P}'$, $\dot{\mbb{Q}}'$, $\dot{\mbb{R}}'$, $\mbb{S}'$, $k'$ and $g'$ be
$\pi ( \mbb{P} )$, $\pi ( \dot{\mbb{Q}} )$, $\pi ( \dot{\mbb{R}} )$, $\pi ( \mbb{S} )$, $\pi [k]$
and $\pi [g]$, respectively.
Note that $\mbb{S} ' = \mbb{P} ' * \dot{\mbb{Q}} ' * \dot{\mbb{R}} '$ in $M'$.
Moreover, $k'$ is an $\mbb{S}'$-generic filter over $M'$, and
$g'$ is the $\mbb{P}'$-generic filter over $M'$ naturally obtained from $k'$.
Let $h'$ be the $( \dot{\mbb{Q}} ' )^{g '}$-generic filter over $M'[ g' ]$
naturally obtained from $k'$, and let $i'$ be the $( \dot{\mbb{R}} ' )^{g' * h'}$-generic
filter over $M' [ g' * h' ]$ naturally obtained from $k'$.

Now, we start to prove that $p^*$ is as desired.
First, we prove that $p^* \in \mbb{P}$.
We only check that $X_{p^* , \gamma}$ is non-stationary in $[ \delta^* ]^\omega$
for some $\gamma < \delta^*$. The other properties are easily checked.

First of all, note that $\pi \rst ( \mcal{H}_{\omega_2} \cap M )$ is the identity map since
$\mcal{H}_{\omega_2} \cap M$ is transitive and that $\pi ( \omega_2 ) = \delta^*$.
Let $\gamma := \pi ( \dot{\gamma} )^{g' * h'} < \pi ( \omega_2 ) = \delta^*$.
Then $\mrm{range} ( \bigcup i' )$ is a club subset of $[ \delta^* ]^\omega$
which does not intersect $\bigcup_{p' \in g'} X_{p' , \gamma} = \bigcup_{p \in g} \pi ( X_{p , \gamma} )$.
Here note that $X_{p , \gamma} \in \mcal{H}_{\omega_2} \cap M$ for all $p \in g$
by the property (ii) of conditions in $\mbb{P}$.
So $\bigcup_{p \in g} \pi ( X_{p , \gamma} ) = \bigcup_{p \in g} X_{p , \gamma} = X_{p^* , \gamma}$.
Hence $X_{p^* , \gamma}$ is non-stationary in $[ \delta^* ]^\omega$

We have shown that $p^* \in \mbb{P}$.
Note that $p^*$ is a lower bound of $g$. Then $p^* \leq p$ since $p \in g$ by (iv).
Suppose $G$ is a $\mbb{P}$-generic filter over $V$ with $p^* \in G$.
Working in $V[G]$, we find a filter $h \subseteq \mbb{Q}$ satisfying (i) and (ii).

Let $M[G]$ denote the collection of $\dot{a}^G$ for all $\mbb{P}$-names $\dot{a} \in M$,
and define $\hat{\pi} : M[G] \to M'[g']$ by $\hat{\pi} ( \dot{a}^G ) := \pi ( \dot{a} )^{g'}$.
It is easy to see that $\hat{\pi}$ coincides with the transitive collapse of $M[G]$
and that $\hat{\pi}$ extends $\pi$.
Let $h$ be the filter on $\mbb{Q}$ generated by $\hat{\pi}^{-1} [ h' ]$.
Then $h$ satisfies (i) since $D_\xi \in M[G]$ and $h' \cap \hat{\pi} ( D_\xi ) \neq \emptyset$
for all $\xi < \omega_1$.
As for (ii), it is easy to see that $\dot{T}_n^h = \ddot{T}_n^k$ for each $n < \omega$.
Then, $h$ satisfies (ii) by (v).
\end{proof}

%%%%%%%%%%%%%%%%%%%%%%%%%%%%%%%%%%%%%%%%%%%%%%%%%%%%%%%%%%%%
%%%%%%%%%%%%%%%%%%%%%%%%%%%%%%%%%%%%%%%%%%%%%%%%%%%%%%%%%%%%

\end{document}